\documentclass[10pt]{amsart}
\usepackage[english]{babel}
\usepackage{amsmath,amssymb,amsthm,amscd}
\usepackage{url}
\usepackage{color,hyperref}
\usepackage{pst-plot}

\theoremstyle{plain}
\newtheorem{theorem}{Theorem}[section]
\newtheorem{proposition}[theorem]{Proposition}
\newtheorem{lemma}[theorem]{Lemma}
\newtheorem{corollary}[theorem]{Corollary}

\theoremstyle{definition}

\newtheorem{remark}[theorem]{Remark}
\newtheorem{example}[theorem]{Example}

\numberwithin{equation}{section}

\newcommand{\Z}{\mathbb{Z}}
\newcommand{\Q}{\mathbb{Q}}
\newcommand{\R}{\mathbb{R}}
\newcommand{\set}[2]{\{#1\,|\ #2\}}
\newcommand{\sub}{\subseteq}
\newcommand{\End}{\mathrm{End}}

\begin{document}
\title[Additively idempotent matrix semirings]{Additively idempotent matrix semirings}

\author[T.~Kepka]{Tom\'{a}\v{s}~Kepka}
\address{Department of Mathematics and Mathematical Education, Faculty of Education, Charles University, M. Rettigov\'{e} 4, 116 59 Prague 1, Czech Republic}
\email{tomas.kepka@pedf.cuni.cz}

\author[M.~Korbel\'a\v{r}]{Miroslav~Korbel\'a\v{r}}
\address{Department of Mathematics, Faculty of Electrical Engineering, Czech Technical University in Prague, Technick\'{a} 2, 166 27 Prague 6, Czech Republic}
\email{korbemir@fel.cvut.cz}

\thanks{The second author acknowledges the support by the bilateral Austrian Science Fund (FWF) project I 4579-N and Czech Science Foundation (GA\v{C}R) project 20-09869L ``The many facets of orthomodularity''  and by the project CAAS CZ.02.1.01/0.0/0.0/16\_019/0000778}



\begin{abstract}

Let $S$ be an additively idempotent semiring and $\mathbf{M}_n(S)$ be the semiring of all $n\times n$ matrices over $S$.
 We characterize the conditions of when the semiring $\mathbf{M}_n(S)$  is congruence-simple provided that the semiring $S$ is either commutative or finite.
 We also give a characterization of when the semiring $\mathbf{M}_n(S)$ is subdirectly irreducible for  $S$ beeing almost integral (i.e., $xy+yx+x=x$ for all $x,y\in S$). In particular, we provide this characterization for the semirings $S$ derived from the pseudo MV-algebras.  
\end{abstract}

\maketitle
\vspace{4ex}

Semirings of endomorphisms of semilattices provide natural examples of additively idempotent semirings. In \cite{jezek2} it was shown that the semiring of all endomorphisms of a non-trivial
semilattice $L$ is always subdirectly irreducible. Moreover, such a  semiring is congruence-simple if and only if the
semilattice $L$ has both the least and the largest elements.

An important case of endomorphisms semirings are the matrix semirings. Here a similar characterization of congruence-simplicity (subdirect irreducibility, resp.) of the matrix semirings is missing (in contrast with the well known fact that the matrix \emph{ring} over a simple \emph{ring}, with a non-trivial multiplication, is simple, too.).

Let $S$ be a semiring $S$ and $\mathbf{M}_n(S)$ be the semiring of all square $n\times n$ matrices over $S$.
For a semiring $S$ with a zero, the case when $\mathbf{M}_n(S)$ is a congruence-simple semiring was characterized in \cite{matrix}.  On the other hand, also semirings without a zero element or without a unity deserve attention, as they naturally appear in mathematics (e.g., as the unbounded distributive lattices). 

Congruence-simple matrix semirings $\mathbf{M}_n(S)$ split into two disjoint classes \cite{simple} - either they are embeddedable into rings or they are additively idempotent. The first of the classes was studied in \cite{matrix}.
In this paper we will investigate the latter class, i.e., the additively idempotent case. 
We provide a full characterization of additively idempotent congruence-simple semirings $\mathbf{M}_n(S)$ under an additional assumption that the semiring $S$ is either commutative or finite (Theorems \ref{commutative} and \ref{finite}). 
It turns out that the only commutative infinite cases of such semirings $S$ are precisely all the unbounded (from above and below) subsemirings of the  well-known tropical semiring $\R(\min,+)$ or, equivalently, of the max-plus semiring $\R(\max,+)$.

A natural generalization of congruence-simple semirings are the subdirectly irreducible ones (that  are the basic building blocks in varieties of semirings according to the well-known Birkhoff's theorem). We may ask therefore, when the semiring  $\mathbf{M}_n(S)$ is subdirectly irreducible. In general such a problem seems to be difficult. However, for additively idempotent semirings $S$ with both a zero and a unity we show that this occurs precisely when $S$ itself is subdirectly irreducible (Theorem \ref{subdirectly_irreducible_4}).

An important class of semirings are those semiring $S$ having a unity $1$ that is aditively absorbing (i.e., $1\cdot x=x\cdot 1=x$ and $x+1=1$ for every $x\in S$). Such semirings are for instance the semiring of all ideals of a ring or the powers-set lattice. These semirings are necessarily additively idempotent and often they are denoted as \emph{integral} \cite{residuated,jipsen}. We study a generalization of the class of integral semirings that will be defined by the conditions $x+x=x$, $xy\leq x$ and $yx\leq x$ for all $x,y\in S$ and call these semirings \emph{almost integral}.

We fully characterize the subdirectly irreducible (congruence-simple, resp.)  matrix semirings $\mathbf{M}_n(S)$, $n\geq 1$, over the almost integral semirings $S$ (Theorems \ref{integral_main}, \ref{lattice} and \ref{subdirectly_irreducible_3}).
 General examples of these cases will be provided in Remarks \ref{construction}, \ref{last_remark} and Examples \ref{example_const} and \ref{non-commutative}.

Finally, we investigate the condition on subdirect irreducibility (congruence-simplicity, resp.) for an important subclass of integral semirings that are closely connected with many-valued \L ukasiewicz logics. In this type of logics, the standard algebraic tools are the MV-algebras and their non-commutative generalizations - the pseudo MV-algebras \cite{dvurecenskij,georgescu,rachunek}. It is known  that  for studying of these object, a semiring approach is helpful, in particular, the theory of semimodules over semirings is used \cite{dinola_2}. Semirings allow an equivalent description of MV-algebras in the form of MV-semirings, or the so called \emph{coupled semirings} in the case of pseudo MV-algebras \cite{dinola_2,shang}. In both of these cases there is a common construction of the corresponding semirings that uses the lattice ordered groups \cite{dvurecenskij,mundici2}. Locally to this paper, we call these semirings \emph{MV-derived} (see Example \ref{example}). 

In this paper we show that the matrix semiring $\mathbf{M}_n(S)$ over an MV-derived semiring $S$ is subdirectly irreducible if and only if $S$ is subdirectly irreducible and that this happens if and only if $S$ has the least non-zero element (Theorem \ref{corollary_GMV}). In particular, the congruence-simple case of $\mathbf{M}_n(S)$ takes place if and only $S$ has two elements.

\section{Preliminaries}

A \emph{semiring} $S$  will be a non-empty set equipped with two associative binary operations, usually denoted as addition $+$ and multiplication $\cdot$ such that the addition is commutative and the multiplication distributes over the addition from both sides.

Let $S=S(+,\cdot)$ be a semiring. For  subsets $A,B$ of $S$ we set  $AB=\set{ab}{a\in A, b\in B}$. A \textit{congruence} on a semiring $S$ is an equivalence relation that is preserved under the addition and multiplication.
The semiring $S$ is called
\begin{itemize}
 \item \emph{congruence-simple} if $S$ has just two congruences;
 \item \emph{subdirectly irreducible (SI)} if $S$ has the least non-identical congruence;
 \item \emph{additively idempotent} if $a+a=a$ for every $a\in S$.
\end{itemize}

An element $w\in S$ is called \emph{right (left, resp.) multiplicatively  absorbing} if $Sw=\{w\}$ ($wS=\{w\}$, resp.);
\emph{multiplicatively absorbing} if $w$ is both right and left multiplicatively absorbing;
 \emph{a unity} if $aw=a=wa$ for every $a\in S$ (such an element will be denoted by $1_S$);
 \emph{bi-absorbing} if $w$ is     multiplicatively absorbing and $a+w=w$ for every $a\in S$ (such an element will be denoted by $o_S$);
 a \emph{zero} if $w$ is multiplicatively absorbing and $a+w=a$ for every $a\in S$ (such an element will be denoted by $0_S$).

For the semiring $\mathbf{M}_n(S)$ of all (square) $n\times n$-matrices over $S$ we consider the usual matrix operations addition $(a_{ij})+(b_{ij})=(a_{ij}+b_{ij})$ and multiplication $(a_{ij})(b_{ij})=(\sum_{k}a_{ik}b_{kj})$.

 For every $a\in S$, let $\overline{a}=(a_{ij})\in \mathbf{M}_n(S)$ be such a matrix that $a_{ij}=a$ for all $i,j$. Clearly, $\overline{a}+\overline{b}=\overline{a+b}$ and $\overline{a}\overline{b}=n\overline{ab}$. Thus $\overline{S}=\set{\overline{a}}{a\in S}$ is a subsemiring of $\mathbf{M}_n(S)$. If $S$ is additively idempotent, then $\overline{S}\cong S$.
 
\

Let us state some main properties of congruence-simple semirings from \cite{simple,matrix}.

\begin{lemma}\cite[Lemma 2.2]{simple}\label{6.0}
 Let $S$ be a congruence-simple semiring such that $|SS|\geq 2$. Then at least one of the following conditions holds:
 \begin{enumerate}
 \item[(i)] $|S|=2$, $S$ is both multiplicatively and additively idempotent and $S$ has no multiplicatively absorbing element.
 \item[(ii)] For every $a,b\in S$ such that $a\neq b$ there are $c,d\in S$ such that $ca\neq cb$ and $ad\neq bd$.
 \end{enumerate}
\end{lemma}

\begin{proposition}\cite[Proposition 3.4]{simple}\label{6.0.1}
 Let $S$ be a semiring. If the semiring $\mathbf{M}_n(S)$ is congruence-simple for some $n\geq 2$ then:
  \begin{enumerate}
  \item[(i)] $S$ is congruence-simple.
  \item[(ii)] $S$ contains no bi-absorbing element.
  \item[(iii)] $|SS|\geq 2$.
  \item[(iv)] $S$ is either additively idempotent or $S$ can be embedded into a ring.
 \end{enumerate}
\end{proposition}

\begin{theorem}\cite[Theorem 2.4]{matrix}\label{6.4}
 Let $w\in S$ be a multiplicatively absorbing element. Then the following conditions are equivalent:
  \begin{enumerate}
  \item[(i)] $\mathbf{M}_n(S)$ is congruence-simple for every $n\geq 1$.
  \item[(ii)] $\mathbf{M}_n(S)$ is congruence-simple for at least one $n\geq 2$.
  \item[(iii)] $S$ is congruence-simple, $S$ contains no bi-absorbing element and $|SS|\geq 2$.
 \end{enumerate}
 If one of these conditions is fulfilled, then $w=0_S\in S$ is a zero.
\end{theorem}

We may now show a further useful property.

\begin{proposition}\label{6.1}
 Let $S$ be a semiring and let the semiring $\mathbf{M}_n(S)$ be congruence-simple for some  $n\geq 2$. Then for every $a,b\in S$ such that $a\neq b$ there are $c,d\in S$ such that $ca\neq cb$ and $ad\neq bd$.
\end{proposition}
\begin{proof}
Let $T=\mathbf{M}_n(S)$.  Assume, for contrary, that for all $a,b,c\in S$ such that $a\neq b$ it holds that $ac=bc$.  By Proposition \ref{6.0.1} (iii) and Lemma \ref{6.0}, we obtain that $|S|=2$ and $S$ is additively idempotent. Hence $|T|=2^{n^2}$. 

Assuming that $|TT|=1$, we obtain that for all $x,y,z,u\in S$ it holds, by the additive idempotency of $S$, that $\overline{xy}=\overline{x}\cdot\overline{y}=\overline{z}\cdot\overline{u}=\overline{zu}$. This implies the equation $xy=zu$ for all $x,y,z,u\in S$ and therefore $|SS|=1$, a contradiction with Proposition \ref{6.0.1}(iii).

Thus $|TT|\geq 2$. Our assumption saying that $ac=bc$ for all  $a,b,c\in S$ implies that  $AC=BC$ for all $A,B,C\in T$. Hence, by Lemma \ref{6.0}, it follows that $|T|=2\neq 2^{n^2}$. This is a contradiction.

In this way we have proved that for every $a,b\in S$ such that $a\neq b$ there is $c\in S$ such that $ac\neq bc$. The rest of the proof follows by an analogous argument.
 \end{proof}

\section{Congruence-simple and subdirectly irreducible  matrix semirings}

In this section we deal with additively idempotent semirings $S$. We characterize the case when the matrix semiring  $\mathbf{M}_n(S)$ is congruence-simple provided that the corresponding ordering on $S$ is downwards directed (Theorem \ref{add_idem}). We also characterize the case when the matrix semiring  $\mathbf{M}_n(S)$ is subdirectly irreducible provided that $S$ has both the zero and the unity (Theorem \ref{subdirectly_irreducible_4}).

We start with a useful property of congruence-simple semirings $\mathbf{M}_n(S)$ that are additively idempotent.

\begin{proposition}\label{7.1}
 Let $S$ be an additively idempotent semiring. If $T=\mathbf{M}_n(S)$ is congru\-ence-simple for some $n\geq 2$ then for every $a,b,e\in S$, $a\neq b$ there are $c,d,f,g\in S$ such that $ca+fe\neq cb+fe$ and $ad+eg\neq bd+eg$.
\end{proposition}
\begin{proof}

 Consider that there are $a,b,e\in S$, $a\neq b$ such that for all $c,f\in S$ is $ca+fe=cb+fe$. Let $A=(a_{ij})\in T$ and $B=(b_{ij})\in T$ be matrices such that $a_{11}=a$, $b_{11}=b$ and $a_{ij}=b_{ij}=e$ for $(i,j)\neq(1,1)$. Then for $C=(c_{ij})\in T$ we have $(CA)_{i1}=c_{i1}a+(\sum_{k=2}^n c_{ik})e=c_{i1}b+(\sum_{k=2}^n c_{ik})e=(CB)_{i1}$ and $(CA)_{ij}=(\sum_{k=1}^n c_{ik})e=(CB)_{ij}$ for $j=2,\dots,n$ and $i=1,\dots, n$. Hence $A\neq B$ and $CA=CB$ for every $C\in T$.

 By Proposition \ref{6.0.1}, we have that $|SS|\geq 2$. Since $S$ is additively idempotent, $S\cong\overline{S}=\set{\overline{a}}{a\in S}$ and $\overline{S}$ is a subsemiring of $T$. Therefore $|TT|\geq |\overline{S}\overline{S}|\geq 2$. By Lemma \ref{6.0}, it now follows that $|T|=2$, a contradiction.

 In this way we have shown that for all  $a,b,e\in S$, $a\neq b$ there are $c,f\in S$ such that $ca+fe\neq cb+fe$.

 By an analogous argument it follows that for  all  $a,b,e\in S$, $a\neq b$ there are $d,g\in S$ such that $ad+eg\neq bd+eg$.
\end{proof}

\begin{corollary}\label{corollary_omega}
 Let $S$ be an additively idempotent semiring with the greatest element $\omega\in S$ of $S(+)$. If the semiring $\mathbf{M}_n(S)$ is congruence-simple for some $n\geq 2$, then the element $\omega$ is neither left nor right multiplicatively absorbing.
\end{corollary}
\begin{proof}
 Assume, for contrary, that $\omega$ is left multiplicatively absorbing. Then for every $a,b\in S$, $a\neq b$ and every $c,f\in S$ we have $ca+f\omega=ca+\omega=\omega=cb+
 \omega=cb+f\omega$. This is a contradiction with Proposition \ref{7.1}.
 
 The other case follows from an analogous argument.     
\end{proof}

The following lemma is essential for finding a special couple of matrices in any non-identical congruence on $\mathbf{M}_n(S)$ provided that the semiring $S$ has the property from the Proposition \ref{7.1}.

\begin{lemma}\label{7.2}
 Let $S$ be an additively idempotent  semiring such that for every $a,b,e\in S$, $a\neq b$ there are $c,d,f,g\in S$ such that $ca+fe\neq cb+fe$ and $ad+eg\neq bd+eg$.

 Let $n\geq 2$ and $i_0,j_0\in \{1,\dots,n\}$. Let $\varrho$ be a congruence on $\mathbf{M}_n(S)$ and $(A,B)\in\varrho$ for some $A=(a_{ij}),B=(b_{ij})\in \mathbf{M}_n(S)$, $A\neq B$.   Then:
\begin{enumerate}
 \item[(i)] If $a_{i_0 j_0}\neq b_{i_0 j_0}$ then there are $A'=(a'_{ij}),B'=(b'_{ij})\in \mathbf{M}_n(S)$, $A'\neq B'$ and  $e'\in S$ such that $(A',B')\in\varrho$ and $a'_{ij}=e'=b'_{ij}$ for all $(i,j)\neq(i_0,j_0)$.
 \item[(ii)] If $a_{ij}=e=b_{ij}$ for some  $e\in S$ and all $(i,j)\neq(i_0,j_0)$ then there are $A'=(a'_{ij}),B'=(b'_{ij})\in \mathbf{M}_n(S)$, $A'\neq B'$ and  $a',b',e'\in S$, $a'\neq b'$  such that $(A',B')\in\varrho$ and $a'_{i j_0}=a'$, $b'_{i j_0}=b'$  and  $a'_{ij}=e'=b'_{ij}$ for all $i$ and $j\neq j_0$.
 \item[(iii)] If there are $a,b,e\in S$  such that $a_{i j_0}=a$, $b_{i j_0}=b$  and  $a_{ij}=e=b_{ij}$ for all $i$ and $j\neq j_0$ then there are $a',b'\in S$, $a'\neq b'$ such that $(\overline{a'},\overline{b'})\in\varrho$.
\end{enumerate}
In particular, there are $a,b\in S$, $a\neq b$ such that $(\overline{a},\overline{b})\in\varrho$.
\end{lemma}
\begin{proof}
(i) Assume the condition for $A$ and $B$. Then there are $e',g\in S$ such that $a_{i_0 j_0}+g\neq b_{i_0 j_0}+g$ and $a_{i j}+e'=e'=b_{i j}+e'$ for all $(i,j)\neq (i_0,j_0)$. Let $E=(e_{ij})$ be a matrix such that $e_{i_0 j_0}=g$ and $e_{ij}=e'$ for all $(i,j)\neq (i_0,j_0)$. Then we set $(A',B')=(A+E,B+E)\in\varrho$.

(ii) Assume the condition for $A$ and $B$. Since $A\neq B$, we have  $a_{i_0j_0}\neq b_{i_0j_0}$  By the assumption on $S$ there are $c,f\in S$ such that $ca_{i_0j_0}+fe\neq cb_{i_0j_0}+fe$. Let $E=(e_{ij})$ be a matrix such that $e_{i i_0}=c$ and $e_{ij}=f$ for all $i$ and $j\neq i_0$. Then we set $(A',B')=(EA,EB)\in\varrho$.

(iii) Assume the condition for $A$ and $B$. Since $A\neq B$, we have $a\neq b$.  By the assumption on $S$ there are $d,g\in S$ such that $a'=ad+eg\neq bd+eg=b'$. Let $E=(e_{ij})$ be a matrix such that $e_{j_0 j}=d$ and $e_{ij}=g$ for all $i\neq j_0$ and $j$. Then, clearly, $(\overline{a'},\overline{b'})=(AE,BE)\in\varrho$.
\end{proof}

\begin{remark}\label{hat}
  Let $S$ be an additively idempotent semiring. For a  congruence $\varrho$ on $\mathbf{M}_n(S)$ we set a relation $\widehat{\varrho}$ on $S$  as $(x,y)\in\widehat{\varrho}$ if and only if $(\overline{x},\overline{y})\in\varrho$ for $x,y\in S$.

  Since the map $\nu: S\to \mathbf{M}_n(S)$, $\nu(x)=\overline{x}$ is an injective semiring homomorphism, the relation  $\widehat{\varrho}$ is a congruence on $S$.
\end{remark}

Now we may present our first result on congruence-simple matrix semirings.

\begin{theorem}\label{add_idem}
 Let $S$ be an additively idempotent semiring such that for all $a,b\in S$ there is $c\in S$ such that $c\leq a$ and $c\leq b$. For $n\geq 2$ the following are equivalent:
\begin{enumerate}
  \item[(i)] $\mathbf{M}_n(S)$ is congruence-simple.
  \item[(ii)] $S$ is congruence-simple and for every $a,b,e\in S$, $a\neq b$ there are $c,d,f,g\in S$ such that $ca+fe\neq cb+fe$ and $ad+eg\neq bd+eg$.
\end{enumerate}
\end{theorem}
\begin{proof}
 (i)$\Rightarrow$(ii): Follows from Proposition \ref{6.0.1} and Proposition \ref{7.1}.

 (ii)$\Rightarrow$(i): Let $\varrho$ be a non-identical congruence on $\mathbf{M}_n(S)$.    By Lemma \ref{7.2}, there are $a,b\in S$, $a\neq b$ such that $(\overline{a},\overline{b})\in\varrho$. Hence, by Remark \ref{hat}, the congruence $\widehat{\varrho}$ on $S$ is non-identical. Since $S$ is congruence-simple, it follows that $\widehat{\varrho}=S\times S$. It follows that $\overline{S}\times\overline{S}\sub\varrho$.

  By the assumption on the ordering on the semiring $S$, we obtain that for every $A\in \mathbf{M}_n(S)$ there are $a,b\in S$ such that $\overline{a}\leq A\leq\overline{b}$. Hence $(A,\overline{b})=(A+\overline{a},A+\overline{b})\in\varrho$. Now it easily follows that there is only a single congruence class of $\varrho$ and therefore $\varrho=\mathbf{M}_n(S)\times \mathbf{M}_n(S)$. Thus, the semiring $\mathbf{M}_n(S)$ is congruence-simple.
\end{proof}

A natural generalization of congruence-simple semirings are the subdirectly irreducible ones. They allow to infer similar properties like for the congruence-simple case (compare, e.g., Proposition \ref{subdirectly_irreducible} and Proposition \ref{6.0.1}.)

The following proposition gives a necessary condition for the semiring $\mathbf{M}_n(S)$ to be subdirectly irreducible.   
 To find an equivalent condition seems to be difficult in general. A partial answer will be given in Proposition \ref{subdirectly_irreducible_2} or in Theorem \ref{subdirectly_irreducible_4}.

\begin{proposition}\label{subdirectly_irreducible}
 Let the semiring $T=\mathbf{M}_n(S)$ be subdirectly irreducible and $n\geq 2$. Then:
 \begin{enumerate}
  \item[(i)]  $S$ is subdirectly irreducible.
  \item[(ii)] $S$ contains no bi-absorbing element.
  \item[(iii)] $|SS|\geq 2$.
 \end{enumerate}
\end{proposition}
\begin{proof}
 (i) For a congruence $\varrho$  on $S$ set a relation $\widetilde{\varrho}$ on $T$ as  $(A,B)\in\tilde{\varrho}$ if and only if $(a_{ij},b_{ij})\in\varrho$ for all pairs $(i,j)$. Clearly, $\widetilde{\varrho}$ is a congruence on $T$ and $\varrho\neq id_{S}$ if and only if $\widetilde{\varrho}\neq id_{T}$.

 Now, let $\mathcal{C}$ be the set of all non-identical congruences on $S$ and $\varrho_{0}=\bigcap\set{\varrho}{\varrho\in \mathcal{C}}$. Clearly, $\widetilde{\varrho}_0=\bigcap\set{\widetilde{\varrho}}{\varrho\in \mathcal{C}}$.

Therefore, if the semiring $T$ is subdirectly irreducible, the congruence $\widetilde{\varrho}_0$ is non-identical. Hence the congruence $\varrho_0$ is non-identical and the semiring $S$ is thus subdirectly irreducible too.

(ii) Assume, on the contrary, that $o\in S$ is a bi-absorbing element. For $i=1,\dots,n$ let $I_i$ be the set of all matrices in $T$ whose $i$-th column contains only the element $o$. Let $\sigma_i$ be a relation on $T$ such that $(A,B)\in\sigma_i$ iff either $A=B$ or $A,B\in I_i$. It is easy to verify that $\sigma_i$ is a non-identical congruence on $T$. Since, clearly, $\bigcap_{i=1}^n\sigma_i=id_{T}$, we obtain a contradiction with $T$ being subdirectly irreducible. It follows that $S$ has no bi-absorbing element.

(iii) Assume, on the contrary, that $|SS|=1$. Then $SS=\{\omega\}$ for some $\omega\in S$ and it holds that $\omega$ is multiplicatively absorbing and $\omega+\omega=\omega^2+\omega^2=\omega(\omega+\omega)=\omega$. Therefore $|TT|=1$. Now, for each couple $(i,j)$ consider the relation $\tau_{(i,j)}$ on $T$ defined as $(A,B)\in \tau_{(i,j)}$ iff the $(i,j)$-entries of $A$ and $B$ are equal. Since $|TT|=1$, the relations $\tau_{(i,j)}$ are congruences on $T$. Now, $\bigcap_{i,j=1}^n\tau_{(i,j)}=id_{T}$, a contradiction with $T$ being subdirectly irreducible. We conclude that $|SS|\geq 2$.
\end{proof}

\begin{proposition}\label{subdirectly_irreducible_2}
 Let $S$ be an additively idempotent  semiring such that for every $a,b,e\in S$, $a\neq b$ there are $c,d,f,g\in S$ such that $ca+fe\neq cb+fe$ and $ad+eg\neq bd+eg$. For $n\geq 2$ the following are equivalent:
 \begin{enumerate}
  \item[(i)] $\mathbf{M}_n(S)$ is subdirectly irreducible.
  \item[(ii)] $S$ is subdirectly irreducible.
 \end{enumerate}
\end{proposition}
\begin{proof}
 (i)$\Rightarrow$(ii):  Follows from Proposition \ref{subdirectly_irreducible}.

  (ii)$\Rightarrow$(i): Assume that $\sigma_0$ is the least non-identical congruence on $S$. Then there are $a_0,b_0\in S$, $a_0\neq b_0$ such that $(a_0,b_0)\in\sigma_0$. Let $\varrho_0$ be the smallest congruence on $\mathbf{M}_n(S)$ containing $(\overline{a_0},\overline{b_0})$. We show that $\varrho_0$ is the least non-identical congruence on $\mathbf{M}_n(S)$.

  Indeed, let $\varrho$ be a non-identical congruence on $\mathbf{M}_n(S)$.  By Lemma \ref{7.2}, there are $a,b\in S$ such that $a\neq b$ and $(\overline{a},\overline{b})\in\varrho$. Hence, by Remark \ref{hat}, the congruence $\widehat{\varrho}$ on $S$ is non-identical.  Therefore, $(a_0,b_0)\in\sigma_0\sub\widehat{\varrho}$ and we obtain that $(\overline{a_0},\overline{b_0})\in \varrho$, by the definition of $\widehat{\varrho}$. Thus $\varrho_0\sub\varrho$ and it follows that the semiring $\mathbf{M}_n(S)$ is subdirectly irreducible.
\end{proof}

\begin{remark}
 A natural question suggests itself: Is the equality $\widehat{\varrho_0}=\sigma_0$ true in the proof of Proposition \ref{subdirectly_irreducible_2}?
\end{remark}

The following characterization provides a surprisingly simple characterization of subdirectly irreducible (additively idempotent) matrix semirings that contain a unity and a zero element. A more detailed characterization of this case will be given for special classes of semiring in Theorems \ref{subdirectly_irreducible_3} and \ref{corollary_GMV}.   

\begin{theorem}\label{subdirectly_irreducible_4}
 Let $S$ be an additively idempotent  semiring with a zero element $0_S$ and a unity $1_S$. For $n\geq 2$ the following are equivalent:
 \begin{enumerate}
  \item[(i)] $\mathbf{M}_n(S)$ is subdirectly irreducible.
  \item[(ii)] $S$ is subdirectly irreducible.
 \end{enumerate}
\end{theorem}
\begin{proof}
It is enough to show that the assumption in Proposition \ref{subdirectly_irreducible_2} if fulfilled for $S$. Let $a,b,e\in S$ be such that $a\neq b$. For $c=d=1_S$ and $f=g=0_S$ we then have $ca+fe=a\neq b=cb+fe$ and $ac+eg=a\neq b=bc+eg$. Now the assertion follows immediately from Proposition \ref{subdirectly_irreducible_2}.
\end{proof}

\section{The commutative and the finite cases}

In this section we fully characterize the case when the matrix semiring $\mathbf{M}_n(S)$ is congruence-simple  provided that the (additively idempotent) semiring $S$ is either commutative or finite. This characterization and classification is based on result from \cite{simple_comm,zumbragel}. A similar classification of subdirectly irreducible semirings seems to be a quite complicated problem and even the cases when such an (additively idempotent) semiring $S$ is commutative or finite are far from being understood (see, e.g., \cite{vechtomov-petrov}).

First we deal with the commutative case of $S$. Let us recall a classification from \cite{simple_comm}.

\begin{theorem}\cite[Theorem 10.1]{simple_comm}\label{comm}
 Let $S(+,\cdot)$ be a commutative congruence-simple semiring that is additively idempotent, $|SS|\geq 2$ and $S$ has no bi-absorbing element. Then $S$ is isomorphic to one of the following cases:
 \begin{enumerate}
  \item[(i)] a two-element lattice $L_2(\vee,\wedge)$;
  \item[(ii)] a subsemiring $A$ of the semiring $\R(\max,+)$ such that $A\cap \R^{+}\neq\emptyset$ and $A\cap \R^{-}\neq\emptyset$.
 \end{enumerate}
 (Notice, that $\R(\min,+)\cong\R(\max,+)$ via the isomorphism $x\mapsto -x$.)
\end{theorem}

Now we are ready to state our result on the commutative case.

\begin{theorem}\label{commutative}
 Let $S$ be a commutative additively idempotent semiring and $n\geq 2$. Then the following conditions are equivalent:
  \begin{enumerate}
  \item[(i)] $\mathbf{M}_n(S)$ is congruence-simple.
  \item[(ii)] $S$ is congruence-simple, $|SS|\neq 1$  and $S$ has no bi-absorbing element.
  \item[(iii)] $S$ is isomorphic either to  $L_2(\vee,\wedge)$ or to a subsemiring $A$ of the semiring $\R(\max,+)$ such that $A\cap \R^{+}\neq\emptyset$ and $A\cap \R^{-}\neq\emptyset$.
 \end{enumerate}
\end{theorem}
\begin{proof}
 (i)$\Rightarrow$(ii): Follows from Proposition \ref{6.0.1}.
 
 (ii)$\Rightarrow$(iii): Follows from Theorem \ref{comm}.

 (iii)$\Rightarrow$(i): If $S$ is isomorphic either to  $L_2(\vee,\wedge)$, the assertion follows from  Theorem \ref{6.4}. 
 Let $S$ be a subsemiring  of $\R(\max,+)$ such that $S\cap \R^{+}\neq\emptyset$ and $S\cap \R^{-}\neq\emptyset$.
 
 We are going to verify the condition (ii) in Theorem \ref{add_idem} for $S$. This means that for all $a,b,e\in S$, $a\neq b$, we need to find $c,f\in S$ such that $\max\{c+a,f+e\}\neq\max\{c+b,f+e\}$. Set $c=a$. Since $S\cap \R^{+}\neq\emptyset$, there is $f'\in S$ such that $f'<0$. Hence there is a positive integer $k$ such that $f=kf'\leq \min\{c+a-e,c+b-e\}$. Therefore $f+e\leq \min\{c+a,c+b\}$ and we  obtain that $\max\{c+a,f+e\}=c+a\neq c+b=\max\{c+b,f+e\}$. Hence $\mathbf{M}_n(S)$ is congruence-simple by Theorem \ref{add_idem}.
\end{proof}

Now we characterize the finite case. For this purpose let us recall a construction and a result from \cite{zumbragel}.

For a finite lattice $L$, $|L|\geq 2$, we denote by  $\End_0(L)(\vee,\circ)$  the semiring of all  join-morph\-isms of $L$ preserving the least element of $L$. For $f,g\in \End_0(L)$ we naturally have $(f\vee g)(x)=f(x)\vee g(x)$ and $(f\circ g)(x)=f(g(x))$ for every $x\in L$. Set $X(L)=\set{f\in \End_0(L)}{|f(L)|\leq 2}$.

\begin{theorem}\label{classification}
\cite[Theorem 5.1]{zumbragel} Let $S$ be a finite additively idempotent semiring, $|S|\geq 3$. Let the greatest element $\omega\in S$ of $S(+)$ be neither left nor right multiplicatively absorbing. Then there are a finite lattice $L$ and a subsemiring $R$ of $\End_0(L)(\vee,\circ)$ with $X(L)\sub R$ such that $S$ is isomorphic to $R$. Conversely, every such a subsemiring $R$ is congruence-simple.
\end{theorem}

Let us state a lemma on the two-element case yet.

\begin{lemma}\label{6.5}
 Let $S$ be an additively idempotent two-element semiring. If $\mathbf{M}_n(S)$ is congruence-simple for some $n\geq 2$ then $S(+,\cdot)\cong L_2(\vee,\wedge)$  where $L_2=\{0,1\}$ is a lattice with $0<1$.
\end{lemma}
\begin{proof}
The additive structure on $S$ is a semilattice. Hence $S=\{a,b\}$ and $a<b$. It follows that $a^2=aa\leq \inf\{ab,ba,bb\}$. By Proposition \ref{6.0.1}, $S$ has no bi-absorbing element. Assuming now that $a^2=b$ we obtain that the element $a$ is bi-absorbing, a contradiction. Therefore $a^2=a$. Similarly, $\sup\{ab,ba,aa\}\leq bb$ implies that $b^2=b$.  

Further, the possibility $ab=b=ba$ means, that the element $b$ is bi-absorbing, a contradiction. Further, the case when $ab=b$ and $ba=a$ implies that $ax=bx$ for every $x\in S$, a contradiction with Proposition \ref{6.1}. Similarly, the assumption that
$ab=a$ and $ba=b$ implies that $xa=xb$ for every $x\in S$, again a contradiction with Proposition \ref{6.1}.

Finally, the remaining case $ab=a=ba$ means that $S$ is a two-element lattice.
\end{proof}

Our result on finite (additively idempotent) congruence-simple matrix semirings is as follows.

\begin{theorem}\label{finite}
 Let $S$ be a finite additively idempotent semiring and $n\geq 2$. Then $\mathbf{M}_n(S)$ is congruence-simple if and only if $S\cong R$, where $R$ is a subsemiring  of $(\End_0(L),\vee,\circ)$ for a finite lattice $L$ and $X(L)\sub R$.
\end{theorem}
\begin{proof}
 Let $\mathbf{M}_n(S)$ be congruence-simple. If $|S|=2$, the assertion follows from Lemma \ref{6.5} and the fact that $L_2\cong \End_0(L_2)$ for the lattice $L_2=\{0,1\}$.
 Assume that $|S|\geq 3$.  By Corollary  \ref{corollary_omega}, the greatest element $\omega$ of $S(+)$ has to be neither left nor right multiplicatively absorbing. Hence, by Theorem \ref{classification}, $S\cong R$, where $R$ is a subsemiring  of $(\End_0(L),\vee,\circ)$ for some finite lattice $L$ and $X(L)\sub R$.

 Conversely, let $S\cong R$, where $R$ is a subsemiring  of $(\End_0(L),\vee,\circ)$ for some finite lattice $L$ and $X(L)\sub R$.
If $|S|=2$ then $S$ is obviously congruence-simple and if $|S|\geq 3$ then $S$ is congruence-simple by Theorem \ref{classification}. Further, $S$ contains a zero element and $|SS|\geq 2$. Hence, by Theorem \ref{6.4}, the semiring $\mathbf{M}_n(S)$ is  congruence-simple for any $n\geq 2$.
\end{proof}

\section{Almost integral semirings}\label{section_integral}

An important subclass of additively idempotent semirings are semirings with a unity that is also the greatest element. In \cite{golan3,golan,golan2} such semirings are called \emph{simple}, in \cite{cao} they are called \emph{incline algebras} and in \cite{residuated,jipsen} they are called \emph{integral}.
Examples of such  semirings are, e.g., the semiring of all ideals of a ring, the powers-set lattice or  MV-semirings (that correspond to MV-algebras, see \cite{dinola}).

In accordance with this terminology (see also \cite{behrens} for the semigroup case) we call a semiring $S(+,\cdot)$  \emph{almost integral}  iff $S$ is additively idempotent and  for all $a,b\in S$ it holds that $ab\leq a$ and $ba\leq a$ (or, equivalently, $a+a=a$ and $ab+ba+a=a$). 

If, moreover, the semiring $S$ possesses the unity we call it \emph{integral} (in accordance with the terminology in \cite{jipsen}).

We fully characterize the subdirectly irreducible (congruence-simple, resp.)  matrix semirings $\mathbf{M}_n(S)$, $n\geq 1$, over the almost integral semirings $S$ (Theorems \ref{integral_main}, \ref{lattice} and \ref{subdirectly_irreducible_3}). General examples of these cases will be provided in Remarks \ref{construction}, \ref{last_remark} and Examples \ref{example_const} and \ref{non-commutative}.

First, let us start with some basic properties of almost integral semirings.

\begin{remark}\label{remark_sub}
 Let $S$ be an additively idempotent semiring.

 (i) Assume that $S$ is almost integral. If we add a new element $u$ to $S$ and extend the semiring operation with $a\cdot u=u\cdot a=a$ and $u+a=a+u=u$ for all $a\in S\cup\{u\}$, the set $\overline{S}=S\cup\{u\}$ becomes an integral semiring with a unity $1_{\overline{S}}=u$ that is the greatest element in $\overline{S}$.
 
 (ii) Let $u\in S$ be such that $u^2=u$. Then the set $S_u=\set{uau}{a\in S,\ uau\leq u}$ is a subsemiring of $S$ with a unity $1_{S_u}=u$ that is the greatest element in $S_u$. Hence $S_u$ is integral. 
\end{remark}

\begin{proposition}\label{sub-ordered_basics}
Let $S$ be an almost integral semiring.
\begin{enumerate}
\item[(i)] If $w\in S$ is a minimal element of $S$ then $w=0_S$ is a zero.
\item[(ii)] If $S$ has a unity $1_S\in S$ then $1_S$ is the greatest element of $S$.
\end{enumerate}
\end{proposition}
\begin{proof}
 (i) By the minimality of $w$, the inequality $aw\leq w$ implies $aw=w$  (and, similarly, $wa=w$) for every $a\in S$. Hence $w$ is multiplicatively absorbing.  Now, we have $a\leq w+a=wa+a\leq a+a=a$ and therefore $w+a=a$ for every $a\in S$. Thus, $w$ is a zero element.

 (ii) Clearly, $a= a\cdot 1_S\leq 1_S$ for all $a\in S$.
\end{proof}

 To obtain a characterization of simple (subdirectly irreducible, resp.) matrix semirings over almost integral semirings, we need, in view of Proposition \ref{6.0.1} and Proposition \ref{subdirectly_irreducible}, first investigate the conditions when an almost integral semiring is simple (subdirectly irreducible, resp.). We will do this in Proposition \ref{integral_main}.

The following remark is easy to verify. In accordance with the notation in  semigroup theory we  denote by $S^1$ the smallest multiplicative monoid containing the semigroup $S(\cdot)$.

\begin{remark}\label{transitive}
 Let $S$ be a semiring with a zero element $0$. Let $a,b\in S$. Set $A=\set{(cad+h,cbd+h)}{c,d\in S^1,\ h\in S}$ and $B=id_S\cup A\cup A^T$, where $A^T$ is the transposed relation. Then $(a,b)\in A$ and the transitive closure of $B$ is the least congruence $\tau$ on $S$ containing $(a,b)$.

 Also, if $(0,e)\in \tau$ for some $e\in S$, $e\neq 0$, then there is $g\in S$, $g\neq 0$ such that $(0,g)\in  A\cup A^T$. Hence there are $c,d\in S^1$ and $h\in S$ such that either $cad+h=g\neq 0=cbd+h$ or $cad+h=0\neq g= cbd+h$.
\end{remark}

\begin{theorem}\label{integral_main}
 Let $S$ be an almost integral semiring. Then the following are equivalent:
 \begin{enumerate}
 \item[(i)] $S$ is subdirectly irreducible.

 \item[(ii)] $S$ has the zero element $0\in S$ and a least element  $e\in S$ such that $e\neq 0$. Also, for all $a,b\in S$, $a\not\leq b$ there are $c,d\in S^1$ such that $cad\neq 0= cbd$.
 \end{enumerate}
 If any of the above conditions is fulfilled, then  the least non-identical congruence $\sigma$ of $S$ is of the form $\sigma=\{(0,e),(e,0)\}\cup id_S$. 
\end{theorem} 
\begin{proof}
(i)$\Rightarrow$(ii) Assume that the semiring $S$ is subdirectly irreducible.
Denote by $S^\ast$ the set of elements of $S$ that are not minimal. Clearly $S^\ast\neq\emptyset$, as  $|S|\geq 2$.

First, for $x\in S$  the relation  $\varrho_x$ on $S$ defined as $(a,b)\in\varrho_x$ if and only if $a+x=b+x$ is a semiring congruence on $S$. Indeed, for  $r\in S$ and $(a,b)\in\varrho_x$ we have $ra+rx=rb+rx$ and, as $rx\leq x$, we obtain that $ra+x=ra+rx+x=rb+rx+x=rb+x$ and, consequently, $(ra,rb)\in\varrho_x$. In a symmetric way we infer that $(ar,br)\in\varrho_x$.
If $x\in S^\ast$ then there is $x'\in S$ such that $x'<x$. Hence $(x',x)\in\varrho_x$ and $\varrho_x$ is a non-identical congruence.

  Since $S$ is subdirectly irreducible, the relation $\varrho=\bigcap\set{\varrho_x}{x\in S^\ast}$ is a non-identical congruence. Hence there are $e,f\in S$, $e\neq f$ such that $e+x=f+x$ for every $x\in S^\ast$. Assuming that $e,f\in S^\ast$ we obtain that $e=e+e=f+e$ and, similarly, $f=f+f=e+f$. This means, that $e=e+f=f$ which is a contradiction.

 Therefore, without loss of generality, it holds that $e\in S^\ast$ and $f$ is a minimal element in $S$. Thus, by Proposition \ref{sub-ordered_basics}(i), $f=0_S$ is a zero element of the semiring $S$.  Hence for every $x\in S^\ast$ we have $x=f+x=e+x$, i.e., $e\leq x$. This means that $e$ is the least element of $S^\ast$.

 Now it is easy to verify that $\sigma=\{(0,e),(e,0)\}\cup id_S$ is a congruence of the semiring $S$. As $\sigma$ is a minimal non-identical congruence and $S$ is subdirectly irreducible it follows that $\sigma$ is the least non-identical congruence.

 To prove the second condition of (ii), pick $a,b\in S$ such that $a\not\leq b$. Then $b<a+b$. Let $\tau$ be the least congruence on $S$ containing $(b,a+b)$. Then $(0,e)\in\sigma\sub\tau$. By Remark \ref{transitive}, it follows that there are $c,d\in S^1$ and $h\in S$ such that either $cbd+h\neq 0=c(a+b)d+h$ or $cbd+h=0\neq c(a+b)d+h$. Since $a<a+b$ implies that $0\leq cbd+h\leq c(a+b)d+h$ it follows that the case $cbd+h=0\neq c(a+b)d+h$ takes place.  From $cbd+h=0$ it follows that $cbd=0=h$ and therefore $cad\neq 0=cbd$.

 This concludes the proof of the first implication.

(ii)$\Rightarrow$(i) First, it is easy to verify that $\sigma=\{(0,e),(e,0)\}\cup id_S$ is a non-identical congruence on $S$.

Further, let $\varrho$ be a congruence on $S$ such that $(a,b)\in\varrho$ for some $a,b\in S$, $a\neq b$. Without loss of generality, we may assume that $a\not\leq b$.  By the assumption, there are $c,d\in S^1$ such that $cad\neq 0= cbd$. Hence there is $g\in S$, $g\neq 0$ such that $(0,g)\in\varrho$. Since $e\leq g$ we obtain that $(e,g)=(0+e,g+e)\in\varrho$. By the transitivity of $\varrho$ it follows that $(0,e)\in\varrho$. Therefore $\sigma\sub \varrho$ and the semiring $S$ is subdirectly irreducible.
\end{proof}

\begin{remark}
 Let us write the condition (ii)  in Theorem \ref{integral_main} that uses $S^1$ in an explicit way:  For all $a,b\in S$, $a\not\leq b$ there are $c,d\in S$ such that either $a\neq 0= b$ or $ca\neq 0= cb$ or $ad\neq 0= bd$ or $cad\neq 0= cbd$.
\end{remark}

Now we are ready to state the main results of this section (Theorems \ref{lattice} and \ref{subdirectly_irreducible_3}).

\begin{theorem}\label{lattice}
 Let $S$ be an almost integral semiring $S$. The following are equivalent:
\begin{enumerate}
  \item[(i)] $\mathbf{M}_n(S)$ is congruence-simple for some $n\geq 2$.
  \item[(ii)] $S$ is congruence-simple and $|SS|\geq 2$.
  \item[(iii)] $S(+,\cdot)\cong L_2(\vee,\wedge)$, where $L_2$ is the two-element lattice.
\end{enumerate}
 If $S$ is congruence-simple then $|S|=2$. 
\end{theorem}
\begin{proof}
The implication (i)$\Rightarrow$(ii) follows from Proposition \ref{6.0.1} and the implication (iii)$\Rightarrow$(i) follows from Theorem \ref{6.4}.

If $S$ is congruence-simple then $S$ is subdirectly irreducible as well. Hence, by Theorem \ref{integral_main}, the semiring $S$ has the zero element $0\in S$, the least element  $e\in S$ such that $e\neq 0$ and the least non-identical congruence $\sigma$ of $S$ is of the form $\sigma=\{(0,e),(e,0)\}\cup id_S$. Since $S$ is congruence-simple, we obtain that $\sigma=S\times S$. Therefore $|S|=2$.

If, moreover, $|SS|\geq 2$ holds, then it follows that $e\cdot e=e$. Hence
 $S(+,\cdot)\cong L_2(\vee,\wedge)$, where $L_2$ is the two-element lattice. In this way we have proved the remaining implication (ii)$\Rightarrow$(iii).
\end{proof}

\begin{remark}\label{congruence_1}
 Let $S$ be a semiring. The relation $\lambda_S$ on $S$, defined as $(a,b)\in\lambda_S$ if and only if $ax=bx$ for all $x\in S$, is a congruence on $S$.
 
 Similarly, the relation $\varrho_S$ on $S$, defined as $(a,b)\in\varrho_S$ if and only if $xa=xb$ for all $x\in S$, is a congruence on $S$.
\end{remark}

\begin{theorem}\label{subdirectly_irreducible_3}
 For an almost integral semiring $S$  and $n\geq 2$ the following are equivalent:
 \begin{enumerate}
  \item[(i)] $\mathbf{M}_n(S)$ is subdirectly irreducible.
  \item[(ii)] $S$ is subdirectly irreducible and for all $a,b\in S$, $a\neq b$ there are $c,d\in S$ such that $cad\neq cbd$.
 \item[(iii)] The semiring $S$ has the zero element $0\in S$ and a least element  $e\in S$ such that $e\neq 0$. Also, for all $a,b\in S$, $a\not\leq b$ there are $c,d\in S$ such that $cad\neq 0= cbd$.
 \end{enumerate}
\end{theorem}
\begin{proof}
  (i)$\Rightarrow$(ii): By Proposition \ref{subdirectly_irreducible}, $S$ is subdirectly irreducible. By Theorem \ref{integral_main}, $S$ has a zero element $0$, a least element $e\in S$ such that $e\neq 0$ and the least non-identical congruence on $S$ is of the form $\{(0,e),(e,0)\}\cup id_S$.  
  
  Now we show that for all $a,b\in S$, $a\neq b$ there is $d\in S$ such that $ad\neq bd$. Assume the contrary, i.e., that the congruence $\lambda_S$ in Remark \ref{congruence_1} is non-identical. By the form of the least non-identical congruence on $S$ we have that $(e,0)\in \lambda_S$ and therefore $ex=0x=0$ for every $x\in S$. For $i=1,\dots,n$ let us define a relation $\tau_i$ on $T=\mathbf{M}_n(S)$ as follows: $(A,B)\in\tau_i$ if and only if $A=B$ or the matrices $A$ and $B$ differ only in their $i$-th columns and any entry of these $i$-th columns is equal either to $e$ or $0$. 
  
Let $A,B,C\in T$ and $(A,B)\in\tau_i$. Since $e$ is the least non-zero element in $S$ it follows that $(A+C,B+C)\in\tau_i$. Further, from the fact that $ex=0$ for every $x\in S$ we obtain that $AC=BC$. Finally, as $xe\leq e$ holds for every $x\in S$, we easily infer that $(CA,CB)\in\tau_i$.

Summing this up, we have shown that $\tau_i$ is a non-identical congruence on $T$. On the other hand, clearly, $\bigcap_{i=1}^n \tau_i= id_T$ and this is a contradiction with the fact that $T=\mathbf{M}_n(S)$ is subdirectly irreducible.

Therefore, indeed, for all $a,b\in S$, $a\neq b$ there is $d\in S$ such that $ad\neq bd$. By an analogous argument we prove that for all $a,b\in S$, $a\neq b$ there is $c\in S$ such that $ca\neq cb$. The assertion of (ii) now immediately follows.
 
   (ii)$\Rightarrow$(iii): By Theorem \ref{integral_main}, for all $a,b\in S$ $a\not\leq b$ there are $c',d'\in S^1$ such that $c'ad'\neq 0=c'bd'$. By the assumption of (ii), there are $c'',d''\in S$ such that $c''c'ad'd''\neq 0=c''c'bd'd''$. Now it is enough to set $c=c''c'\in S$ and $d=d'd''\in S$. The rest follows from Theorem \ref{integral_main}.

 (iii)$\Rightarrow$(i): By Theorem \ref{integral_main}, $S$ is subdirectly irreducible. It is enough to show that the assumption in Proposition \ref{subdirectly_irreducible_2} if fulfilled for $S$. Let $a,b,e\in S$ be such that $a\neq b$. Then without loss of generality $a\not\leq b$. By the assumption of (iii), there are $c,d\in S$ such that $cad\neq cbd$, in particular, we have $ca\neq cb$. Now for $f=0$ it holds that 
 $ca+fe=ca\neq cb=cb+fe$. Symmetrically we obtain the remaining condition. Now the assertion follows immediately from Proposition \ref{subdirectly_irreducible_2}.
\end{proof}

To show that the class of subdirectly irreducible (SI) almost integral semiring is broader than the class of SI integral semirings we provide several examples (see Remarks \ref{construction}, \ref{last_remark} and Examples \ref{example_const} and \ref{non-commutative}). In particular, let us note that there are SI almost integral semirings whose matrix semirings are \emph{not} SI (see Remark \ref{last_remark}). While for SI integral semirings this cannot happen, as they possess the unity that assures the condition (ii) in Theorem \ref{subdirectly_irreducible_3}.

To find examples of SI almost integral semirings we first investigate their structure of  in a more subtle way. 

\begin{proposition}\label{integral_unity}
Let $S$ be an almost integral semiring without the unity. Then the following are equivalent:
\begin{enumerate}
\item[(i)] $S$ is subdirectly irreducible.
\item[(ii)] $\overline{S}$  is subdirectly irreducible (see Remark \ref{remark_sub}(i)).
\end{enumerate}
 
\end{proposition}
\begin{proof}
(i)$\Rightarrow$(ii):  We only need to check the condition (ii) in Theorem \ref{integral_main} for new pairs of elements, i.e., for $a=1_{\overline{S}}\notin S$ and $b\in S$. Since $b$ is not a unity in $S$, there is $g\in S$ such that either $bg\neq g$ or $gb\neq g$, in particular, either $bg<g$ or $gb<g$. Assume the case $bg<g$. As the semiring $S$ is subdirectly irreducible, there are $c,d\in S^1$ such that $cbgd=0\neq cgd$. Now, for $c'=c$ and $d'=gd$ we have $c'bd'=0\neq c'ad'$. The other case is similar.
   
(ii)$\Rightarrow$(i): Follows immediately from Theorem \ref{integral_main}.
\end{proof}

\begin{proposition}\label{si_proposition}
 Let $S$ be an almost integral subdirectly irreducible semiring with a zero element $0$ and a least non-zero element $e$.
 \begin{enumerate}
 \item[(i)] If $|S|\geq 3$ then $e^2=0$. In particular, such a semiring $S$ is not multiplicatively idempotent (cf. \cite{vechtomov-petrov}).
 \item[(ii)] Let $S$ be commutative. If $a\in S$ is such that $ae\neq 0$ then $a$ is the greatest element of $S$. If $S$ has the unity $1\in S$ then $x+y=1$ implies $x=1$ or $y=1$ for all $x,y\in S$.
 \end{enumerate}
\end{proposition}
\begin{proof}
 (i) If $|S|\geq 3$ then there is $a\in S$ such that $e<a$. Then, by Theorem \ref{integral_main} there are $c,d\in S^1$ such that $cad\neq 0=ced$. Hence, as $e\neq 0$,  it follows that either $0\neq c\in S$ or $0\neq d\in S$. If $0\neq c\in S$ then $c\geq e$ (and the same holds for $d$).  
 
 Assume, for contrary, that $e^2\neq 0$. Then $e^2=e$ and, consequently, $e^3=e$. Now, if $0\neq c\in S$ and $d\notin S$ or if $c\notin S$ and $0\neq d\in S$, then we have $e=e^2\leq ced=0$, a contradiction. Finally, if $0\neq c\in S$ and $0\neq d\in S$ then we obtain that $e=e^3\leq ced=0$, again a contradiction. 
 
 Therefore, $e^2=0$.

 (ii) Assume that $S$ is commutative. Let $a\in S$ be such that $ae\neq 0$. Since $ae\leq e$ we obtain that $ae=e$. Assume, for contrary, that $a$ is not the greatest element of $S$. Then there is $b\in S$ such that $b\not\leq a$. By Theorem \ref{integral_main}, there are $c,d\in S^1$ such that $(cd)b=cbd\neq 0=cad$. Hence $cd\in S\setminus\{0\}$ and therefore $e\leq cd$. It follows that $e=ae\leq a(cd)=0$, a contradiction.
 
 Let, moreover, $S$ have a unity. Let $x+y=1$ for some $x,y\in S$. Then $xe+ye=1\cdot e=e$ and therefore either $xe\neq 0$ or $ye\neq 0$. By the previous part we obtain that either $x$ or $y$ are the greatest element of $S$. Hence either $x=1$ or $y=1$.
\end{proof}

The properties stated in Proposition \ref{si_proposition} are an inspiration for the following construction.

\begin{remark}\label{construction}
 Let $S$ be an almost integral semiring with a zero element $0$ and let for all $a,b\in S$, $a\not\leq b$, there be $c,d\in S^1$ such that $cad\neq0=cbd$. Assume, moreover, that if $a\in S$ is a minimal element of $S\setminus\{0\}$ then $a^2=a$.
 
 Add a new element $e$ to $S$ in the way that $e$ becomes the least element of $(S\setminus\{0\})\cup\{e\}$.  The multiplicative operation on $S$ we extend by $e\cdot a=a\cdot e=0$ for every $a\in S\cup\{e\}$.
 
Then $S'$ is an almost integral subdirectly irreducible semiring.
\begin{proof}
 Indeed, we only need to verify condition (iii) in Proposition \ref{integral_main}, as the rest is easy. Let therefore $a,b\in S'$ be such that $a\not\leq b$. Among all the possibilities of $a$ and $b$, only the case when $0\neq a\in S$ and $b=e$ needs to be justified. 
 
 Assume first that $a$ is not a minimal element of $S\setminus\{0\}$. Then there is $a'\in S$ such that $0<a'<a$, hence, $a\not\leq a'$. By the assumption on $S$ there are $c,d\in S^1$ such that  $cad\neq0=ca'd$. Consequently, either $c\in S$ or $d\in S$ and, therefore, we obtain that $cbd=ced=0$. 
 
For the second case, let $a$ be a minimal element of $S\setminus\{0\}$. By the assumption on $S$ we have $a^2=a$ and we may set $c=d=a$. Then $cad=a^3=a\neq 0$ while $cbd=ced=0$ and we are done. 
\end{proof} 
\end{remark}

\begin{example}\label{example_const}
The following examples of a semiring $S$ fulfill the assumption from Remark \ref{construction}. Hence $S$ allows an extension  to a subdirectly irreducible semiring $S'$ according to Remark \ref{construction}.

 (i) $S(+,\cdot)=B(\vee,\wedge)$ where $B(\vee,\wedge, ',0,1)$ is a Boolean algebra.  For $a,b\in B$, $0\neq a\setminus b=a\wedge b'$ it is enough to set $c=d=b'$. Notice that  the subdirectly irreducible semiring $S'$ has the greatest element, but it  is \emph{not} a unity.
 
 (ii) $S(+,\cdot)=L(\cup,\cap)$ where $L$ is the set of all finite subsets of an infinite set $X$. For $a\sub X$ and $b\sub X$, $0\neq a\setminus b$, it is enough to set $c=d=a\setminus b$. Notice that the subdirectly irreducible semiring $S'$ is \emph{without} the greatest element.
\end{example}

\begin{remark}\label{last_remark}
 Let us point out that there are subdirectly irreducible almost integral semirings $S'$ (see Example \ref{example_const}) such that the semiring $\mathbf{M}_n(S')$ is \emph{not} subdirectly irreducible, while  for the extended semiring $\overline{S'}=S'\cup\{1\}$ (see Remark \ref{remark_sub}(i)), the semiring $\mathbf{M}_n(\overline{S'})$ is subdirectly irreducible. 
 
 Indeed, by the Example \ref{example_const} and the construction Remark \ref{construction} there is a subdirectly irreducible almost integral semiring $S'$ with the smallest non-zero element $e\in S'$ such that $ex=xe=0=0x=x0$ for all $x\in S'$. By Theorem \ref{subdirectly_irreducible_3}(ii), the matrix semiring $\mathbf{M}_n(S')$ is not subdirectly irreducible. On the other hand, for the subdirectly irreducible (almost) integral semiring $\overline{S'}$ with a unity (see Remark \ref{remark_sub}(i)), the semiring $\mathbf{M}_n(\overline{S'})$ is subdirectly irreducible, by Theorem \ref{subdirectly_irreducible_3}(ii).
\end{remark}

\section{Connection to pseudo MV-algebras}

An important subclass of integral semirings are the semirings arising from MV-algebras and \emph{pseudo} MV-algebras \cite{dvurecenskij,georgescu,rachunek,shang}. These algebras are closely connected to many-valued logics. By a famous Mundici’s
representation theorem \cite{mundici2} every MV-algebra can be represented as an interval in some abelian $\ell$-group. In \cite{dvurecenskij} this correspondence was further generalized by  expressing every pseudo MV-algebra as an interval  $[0,u]$, with $u>0$, in a  generally non-abelian $\ell$-group $G(\ast,^{-1},0, \wedge,\vee)$ (for the definition of an $\ell$-group see Example \ref{example}). Using this representation, to every pseudo MV-algebra can be assigned a semiring $S$ as in  Example \ref{example}. Let us, for short, call such semirings \emph{MV-derived}. 

These semirings are \emph{lattice-ordered}  which means that such a semiring $S(+,\cdot)$ is additively idempotent, the ordering $\leq$ has a structure of a lattice and there holds that $a+b=a\vee b$ and $a\cdot b\leq a\wedge b$ for all $a,b\in S$. For the details, definitions and theory  we refer the reader to \cite{golan}.  Clearly, every lattice-ordered semiring is (almost) integral. 

\begin{example}\label{example}
 Let $G(\ast,^{-1},0, \wedge,\vee)$ be a lattice-ordered group ($\ell$-group, for short). This means that
 \begin{itemize}
  \item $G(\ast,^{-1},0)$ is a group with a neutral element $0$.
    \item $G(\wedge,\vee)$ is a lattice.
  \item the operation $\ast$ distributes over $\vee$  and over $\wedge$ from both sides.
 \end{itemize}
 (Let us notice that we have denoted by $\ast$ the operation in the group $G$ to keep the operations $+$ and $\cdot$ only for semirings, while the notation $0$ of the neutral element in $G$ is chosen due to the usual convention used for MV-algebras.)
 
Now, choose $u\in G$ such that $u>0$ and consider the interval $$S=[0,u]=\set{a\in G}{ 0\leq a\leq u}.$$
For $a,b\in S=[0,u]$ we set
$$a\cdot b=(a\ast u^{-1}\ast b)\vee 0\ .$$

Then $S(\vee,\cdot)$ is an additively idempotent semiring (see, e.g., \cite{dvurecenskij,shang}) with the zero element $0$ and the unity $u$ (that, obviously, is the greatest element in $S$). Hence, by Remark \ref{remark_sub}, the MV-derived semiring $S$ is integral.
\end{example}

The following two Lemmas \ref{mv_basic} and \ref{l-groups_basic} will help us to characterize the subdirectly irreducible MV-derived semirings.

\begin{lemma}\label{mv_basic}
 Let $S$ be an MV-derived semiring. Then for all $a,b\in S$, $a\not\leq b$, there are $c,d\in S$ such that $cad\neq 0=cbd$. 
\end{lemma}
\begin{proof}
 Assume the construction of $S=[0,u]$ used in Example \ref{example}. Let $a,b\in S$, $a\not\leq b$.  We set $c=b^{-1}\ast u\in S$ and $d=u$. Since $u$ is a unity in $S$ we have $cbd=cb=(b^{-1}\ast u\ast u^{-1}\ast b)\vee 0=0\vee 0=0$. From  $a\not\leq b$ it follows that $b^{-1}\ast a\not\leq 0$ and therefore $cad=ca=(b^{-1}\ast u\ast u^{-1}\ast a)\vee 0=(b^{-1}\ast a)\vee 0\neq 0$.
\end{proof}

\begin{example}\label{non-commutative}
 Let $\mathcal{G}(\R)$ be the group of all continuous strictly increasing bijective maps $\R\to\R$ with the group operation defined by composition of maps. Define the partial order on $\mathcal{G}(\R)$ by $f\leq g$ if and only if $f(x)\leq g(x)$ for every $x\in\R$. This ordering forms a lattice with operations $(f\wedge g)(x)=\min\{f(x),g(x)\}$ and $(f\vee g)(x)=\max\{f(x),g(x)\}$ for every $x\in\R$.
 
 Now, $\mathcal{G}(\R)$ is a non-commutative lattice-ordered group with neutral element $o=id_\R$. If $u\in\mathcal{G}(\R)$ is such that $o<u$, then the interval $[o,u]$ has no minimal element different from $o$.
 
 Let $S=[o,u]$ be the corresponding MV-derived semiring. By Remark \ref{construction} and Lemma \ref{mv_basic}, the extended semiring $S'$ is non-commutative, almost integral and subdirectly irreducible. 
\end{example}

\begin{lemma}\label{l-groups_basic}
 Let $G(\ast,^{-1},0, \wedge,\vee)$ be a lattice-ordered group with a unit element $0$. Let $u\in G$ be such that $u>0$ and  the interval $[0,u]=\set{x\in G}{0\leq x\leq u}$ has the least element different from $0$. Then the ordering on the interval $[0,u]$ is total.
\end{lemma}
\begin{proof}
 Let $e$ be the least element in $[0,u]$ different from $0$.  Assume for contrary that there are $a,b\in [0,u]$ such that $a\not\leq b$ and $b\not\leq a$. Set $c=a\wedge b$. Then $c<a$ and $c<b$ and $0=c\ast c^{-1}=(a\wedge b)\ast c^{-1}=(a\ast c^{-1})\wedge(b\ast c^{-1})$.
 From $0\leq c$ and $c<a$ we obtain that  $a\ast c^{-1}=a\ast c^{-1}\ast 0 \leq a\ast c^{-1}\ast c=a\leq u$ and $0=c\ast c^{-1}\leq a\ast c^{-1}$. Hence $a\ast c^{-1}\in[0,u]$ and, similarly, $b\ast c^{-1}\in[0,u]$.
 
 Further, since $a\ast c^{-1}\neq 0$ and $b\ast c^{-1}\neq 0$ we have $e\leq a\ast c^{-1}$ and $e\leq b\ast c^{-1}$. It follows that $e=e\wedge e\leq (a\ast c^{-1})\wedge(b\ast c^{-1})=0$, a contradiction.
 
 Therefore, the ordering on the interval $[0,u]$ is total.
\end{proof}

Let us notice, that elements with properties from Lemma \ref{l-groups_basic} are studied in theory of $\ell$-groups: an element $a>0$ is called \emph{basic} iff the interval $[0,a]$ is totally ordered. 

Now we are ready to state the main result of this section.

\begin{theorem}\label{corollary_GMV}
 Let  $S$ be an MV-derived semiring (see Example $\ref{example}$) and $n\geq 1$. Then:
 \begin{enumerate}
  \item[(i)] $\mathbf{M}_n(S)$ is subdirectly irreducible if and only if in $S$ there is the least element $e\in S$ such that $e\neq 0$.
  \item[(ii)] $\mathbf{M}_n(S)$   is congruence-simple if and only if $|S|=2$.
 \end{enumerate}
 If $S$ is subdirectly irreducible, then $S$ is totally ordered.
\end{theorem}
\begin{proof}
(i) If  $\mathbf{M}_n(S)$ is subdirectly irreducible for some $n\geq 1$, then $S\cong \mathbf{M}_1(S)$ is subdirectly irreducible and $S$ has the least non-zero element, by Theorem \ref{subdirectly_irreducible_3}. Conversely, if  $S$ has the least non-zero element then, by Lemma \ref{mv_basic} and  Theorem \ref{subdirectly_irreducible_3}, the semiring $\mathbf{M}_n(S)$ is subdirectly irreducible for any $n\geq 1$.
Moreover, by Lemma \ref{l-groups_basic}, the semiring $S$ is totally ordered.

(ii) Follows immediately from Theorem \ref{lattice}.
\end{proof}

\begin{example}
 Let $G=\set{\left(\begin{smallmatrix}\alpha &\beta \\ 0 & 1\end{smallmatrix}\right)}{\alpha,\beta\in\Q,\ \alpha>0}$ be the multiplicative group with a total order defined by $\left(\begin{smallmatrix}\alpha &\beta \\ 0 & 1\end{smallmatrix}\right)\geq \left(\begin{smallmatrix}\alpha' &\beta' \\ 0 & 1\end{smallmatrix}\right)$ if and only if either $\alpha>\alpha'$ or both $\alpha=\alpha'$ and $\beta\geq \beta'$ hold. It is easy to verify that with this ordering is $G$ an $\ell$-group. Denote by $E$ the identical matrix in $G$.
 
Assume now the direct product of groups $H=G\times\Z$, where $\Z$ has the usual additive operation, with the lexicographical ordering defined for every $(A,k), (A',k')\in H$ by $(A,k)\geq (A',k')$ if and only if either $A> A'$ or both $A=A'$ and $k\geq k'$ hold. Again, this total order  makes $H$ into a non-commutative $\ell$-group. Clearly, the element $e:=(E,1)$ is the least element greater than the unity $o:=(E,0)$ in $H$.

Now, put  $u:=\left(\left(\begin{smallmatrix}3 & 0 \\ 0 & 1\end{smallmatrix}\right),0\right)$. Then the MV-derived semiring $S=[o,u]$ is subdirectly irreducible, by Theorem \ref{corollary_GMV}, but \emph{non-commutative}.  

Indeed, for instance, the elements $a=\left(\left(\begin{smallmatrix}2 & 0 \\ 0 & 1\end{smallmatrix}\right),0\right)\in S$ and $b=\left(\left(\begin{smallmatrix}2 & 1 \\ 0 & 1\end{smallmatrix}\right),0\right)\in S$ do not commute in $S$. Moreover, the element $e\in S$ does not commute with $b\in S$ in $S$. This shows that the construction in Remark \ref{construction} captures only some of the cases of the almost integral subdirectly irreducible semirings.
\end{example}

\end{document}